\documentclass[10pt,reqno]{amsart}
\usepackage{tikz}
\usepackage{comment}%
\usepackage{blindtext}
\usepackage{subcaption}
\usepackage{graphicx}
\usepackage{subfiles}
\usetikzlibrary{quotes}
\usetikzlibrary{arrows}
\usepackage{qtree}
\usepackage{ mathdots }

\tikzset{
    vertex/.style={circle,draw,minimum size=1.5em},
    edge/.style={->,> = latex'}
}
\usepackage{amsthm,amsmath,amssymb}
\usepackage[english]{babel}
\usepackage[utf8]{inputenc}
\usepackage{amsthm}
\usepackage{array}
\usepackage{framed}
\usepackage{pgfplots}
\usepackage{float}
\usepackage{graphicx}
\usepackage{amssymb}
\usepackage[inline]{enumitem}
\usepackage{thmtools}

\usepackage{array, makecell} 

\usepgfplotslibrary{fillbetween}
\pgfplotsset{compat=1.17}

\usepackage{mathtools}
\usepackage[colorlinks=true,backref=page,bookmarksopen=true]{hyperref} 
\usepackage[capitalise]{cleveref} 
\usepackage[font=small]{caption}

\numberwithin{equation}{section}     

\declaretheorem[numberwithin=section]{theorem}  
\declaretheorem[sibling=theorem]{corollary} 
\declaretheorem[sibling=theorem]{lemma}

\declaretheorem[sibling=theorem]{question}
\declaretheorem[sibling=theorem, style=remark]{note}

\declaretheorem[sibling=theorem, style=definition]{definition}

\usepackage{xcolor}
\hypersetup{
    colorlinks,
    linkcolor={red!50!black},
    citecolor={blue!50!black},
    urlcolor={blue!80!black}
}
\makeatletter
\NewCommandCopy\@@pmod\pmod
\DeclareRobustCommand{\pmod}{\@ifstar\@pmods\@@pmod}
\def\@pmods#1{\mkern4mu({\operator@font mod}\mkern 6mu#1)}
\makeatother

\renewcommand*{\backref}[1]{}
\renewcommand*{\backrefalt}[4]{\tiny 
  \ifcase #1 (\textbf{NOT CITED.})%
  \or    (Cited on page~#2.)%
  \else   (Cited on pages~#2.)%
  \fi}
  
 \tikzstyle{important line}=[very thick]
 \tikzstyle{information text}=[rounded corners,fill=red!10,inner sep=1ex]
 
\usetikzlibrary{angles,quotes,calc,shapes,positioning,intersections}
 \tikzstyle{important line}=[very thick]
 \tikzstyle{information text}=[rounded corners,fill=red!10,inner sep=1ex]
 
\usetikzlibrary{angles,quotes,calc,shapes,positioning,intersections}
\NewDocumentCommand\Cycle{O{} m m m O{} m}{%
  \draw[->,> = latex'][#1](#2.{#3+asin(#6/(#4*1.41))}) arc (180+#3-45:180+#3-45-270:#6/2) #5;
}
\tikzset{
    vertex/.style={circle,draw,minimum size=1.5em},
    edge/.style={->,> = latex'}
}

\def\env@matrix{\hskip -\arraycolsep
  \let\@ifnextchar\new@ifnextchar
  \array{*\c@MaxMatrixCols c}}
  \makeatletter
\renewcommand*\env@matrix[1][*\c@MaxMatrixCols c]{%
  \hskip -\arraycolsep
  \let\@ifnextchar\new@ifnextchar
  \array{#1}}
\makeatother

\DeclareMathOperator{\tr}{tr}
\DeclareMathOperator{\JSR}{JSR}

\makeatletter
\newcommand*{\rom}[1]{\expandafter\@slowromancap\romannumeral #1@}
\makeatother

\newcommand{\tribar}[1]{\mathopen{| {\kern -1.5pt} | {\kern -1.5pt} |} {#1}
\mathclose{| {\kern -1.5pt} | {\kern -1.5pt} |}}

\newcommand{\arxiv}[1]{\href{http://arxiv.org/abs/#1}{arXiv:{#1}}}
\newcommand{\doi}[1]{\href{http://dx.doi.org/#1}{\tt DOI}} 
\newcommand{\directlink}[1]{\href{#1}{\tt URL}}
\newcommand{\MRev}[1]{\href{https://mathscinet.ams.org/mathscinet-getitem?mr=#1}{\tt MR}} 
\newcommand{\Zbl}[1]{\href{https://zbmath.org/?q=an:#1}{\tt Zbl}} 

\begin{document}
\title[Partial classification of SMPs]{Partial classification of spectrum maximizing products for pairs of $2\times2$ matrices}

\author{Piotr Laskawiec}
\address{Department of Mathematics, The Pennsylvania State University}
\email{\href{mailto:ppl5146@psu.edu}{ppl5146@psu.edu}}

\date{June 2024}


\begin{abstract}
Experiments suggest that typical finite sets of square matrices admit spectrum maximizing products (SMPs): that is, products that attain the joint spectral radius (JSR). Furthermore, those SMPs are often combinatorially "simple." In this paper, we consider pairs of real $2 \times 2$ matrices. We identify regions in the space of such pairs where SMPs are guaranteed to exist and to have a simple structure. We also identify another region where SMPs may fail to exist (in fact, this region includes all known counterexamples to the finiteness conjecture), but nevertheless a Sturmian maximizing measure exists. Though our results apply to a large chunk of the space of pairs of $2 \times 2$ matrices, including for instance all pairs of non-negative matrices, they leave out certain "wild" regions where more complicated behavior is possible.
\end{abstract}

\maketitle
\section{Introduction}
Given a set of square matrices it is a common problem to study the maximal growth rate of products from that set. To this end we introduce the following definition:
\begin{definition}
Let $\mathcal{A}$ be a bounded set of real $d\times d$ matrices.
The \emph{joint spectral radius} (JSR in short) of $\mathcal{A}$ is given by: 
\begin{equation}\label{e.def_JSR}
\JSR(\mathcal{A}) \coloneqq \lim_{k\to\infty}\sup_{\Pi \in \mathcal{A}^{*k}}\|\Pi\|^{\frac{1}{k}} \, ,
\end{equation}
where $\mathcal{A}^{*k}$
 denotes the collection of all products of elements of~$\mathcal{A}$ with length $k$ and $\|\cdot\|$ is any matrix norm.
\end{definition}
The study of this quantity has been done since the 1960's, beginning with Rota and Strang \cite{rota-strang} who introduced it. For any $k$ and any matrix norm $\|\cdot\|$, the joint spectral radius satisfies the \emph{three members inequalities} (See \cite{Jungers} for more details):
\begin{equation}\label{three member}
\sup_{\Pi \in \mathcal{A}^{*k}}\rho(\Pi)^\frac{1}{k}\leq \JSR(\mathcal{A})\leq \sup_{\Pi \in \mathcal{A}^{*k}}\|\Pi\|^{\frac{1}{k}} \, ,
\end{equation}
where $\rho$ denotes the spectral radius.
The upper bound converges to the $\JSR$ by definition. If we take the $\limsup$ the lower bound also converges by the following theorem of Berger-Wang:
\begin{theorem}\label{e.BW}\cite{berger-wang} For a bounded set of matrices $\mathcal{A}$
\begin{equation}
\JSR(\mathcal{A}) = \limsup_{k\to\infty}\sup_{\Pi \in \mathcal{A}^{*k}}\rho(\Pi)^\frac{1}{k} \,.
\end{equation}
\end{theorem} 
In this paper, we will focus our attention on products that attain the joint spectral radius in a particular way. We only consider primitive products which refers to a product that is not a power of some shorter product.
\begin{definition}
Let $\Pi\in \mathcal{A}^{*k}$ be a primitive product of length $k$.
We say that $\Pi$ is a \emph{spectrum maximizing product} (\emph{SMP} in short) if
\begin{equation}
\rho(\Pi)^\frac{1}{k}=\JSR(\mathcal{A}) \, .
\end{equation}
\end{definition} It was previously speculated that every finite set of matrices has such a product. This was disproven \cite{finiteness}, but nonetheless it is conjectured that for a generic set this property holds. We ask a related question (in the simplest case):
\begin{question}\label{e.question}
What products can be (relatively unique) SMPs for a pair of $2\times 2$ real matrices?
\end{question}
\begin{note}
See \cite{BL} for the definition of relative uniqueness. This condition excludes trivial cases such as pairs of the form $(A,A)$, where every primitive product is an SMP.
\end{note}
\begin{note}
If we do not restrict the dimension of matrices, any primitive product is a unique (up to cyclic permutations) SMP for some set of matrices of large enough dimension. The proof will appear in \cite{thesis}, although this result is considered folklore.
\end{note}
Although SMPs do not necessarily exist, it can be shown that, informally, there always exist asymptotically optimal stationary infinite products, which can be seen as coming from maximizing measures (See \cref{section 2} for the details). This leads to a modified and more general question.
\begin{question}
What measures can be maximizing?
\end{question}

Numerically, SMPs are often Sturmian (See \cref{section 9} for the definition and properties) and
there are theoretical results in that direction. In \cite{Panti,Kozyakin,MS} families of pairs of matrices for which SMPs (if they exist) are always Sturmian are studied. However, it is not clear, quantifiably speaking, how common these pairs are. We will attempt to at least partially answer this question for pairs in dimension two.
Often times we will assume our matrices are not \emph{reducible}:
\begin{definition}
We say a pair of matrices $(A,B) \in M_2(\mathbb{R})^2$ is reducible if the matrices are simultaneously triangularizable. Otherwise, we say it is irreducible. 
\end{definition}The reducible case is not only rare but also not interesting. If $(A,B)$ is reducible, either only $A$ or only $B$ is an SMP, or all primitive products are SMPs. Therefore, we will omit this case without loss of generality.

In this paper, we will cover a part of irreducible pairs in $ M_2(\mathbb{R})^2$ by 4 regions, $\mathcal{R}_{cross},\mathcal{R}_{mix},\mathcal{R}_{neg},\mathcal{R}_{copar}$ where we can describe the SMPs more precisely. We have the following descriptions of the regions, where $\mathcal{D}$ denotes the set of real diagonalizable matrices:
\begin{align*}
    (A,B) \in \mathcal{R}_{cross}&\iff \begin{cases}
   (A,B) \in \mathcal{D}^2 \\
   \det(AB-BA)>0
    \end{cases}
    \\
    (A,B) \in \mathcal{R}_{mix} &\iff
        \det(A)\det(B)\leq 0
    \\
    (A,B) \in \mathcal{R}_{neg} &\iff \begin{cases}
        \det(A)<0\\
        \det(B)<0
    \end{cases}\\
    (A,B)\in \mathcal{R}_{copar}&\iff\begin{cases}
        (A,B) \in GL^+(2, \mathbb{R})^2\cap \mathcal{D}^2\\
        \det(AB-BA)<0\\
        |\tr{AB}|>\frac{1}{2}|\tr{A}\tr{B}|\\
        \tr{AB}\tr{A}\tr{B}>0
    \end{cases}
\end{align*}

\begin{note}
All the regions are invariant under independent re-scaling of both matrices and swapping of the two matrices.
\end{note}
We are now ready to state the main theorem.
\begin{theorem}\label{intro thm}
We have the following classification of SMPs in these regions:
\begin{enumerate}[label=(\alph*)]

    \item If $(A,B) \in \mathcal{R}_{cross}$, then an SMP exists and all SMPs are either $A$ or $B$.
    \item If $(A,B) \in \mathcal{R}_{mix}$, then there is an SMP of the form $A,B$ or $AB^n$ (if $\det(A)\leq0\leq\det(B)$) or $A^nB$ (if $\det(B)\leq0\leq \det(A)$)
    \item If $(A,B) \in \mathcal{R}_{neg}$, then $A,B$ or $AB$ is an SMP.
    \item If $(A,B) \in \mathcal{R}_{copar}$, then if an SMP exists it is unique (up to cyclic permutations) and it is Sturmian. There always exists exactly one Sturmian maximizing measure.
\end{enumerate}
Furthermore, Lebesgue almost every pair $(A,B) \in \mathcal{R}_{cross}\cup\mathcal{R}_{mix}\cup\mathcal{R}_{neg}$ has a unique SMP up to cyclic permutations. Therefore, generically, the theorem describes all the SMPs in each region.
\end{theorem}
\subsection{Some notes on the regions}
\begin{enumerate}[label=(\roman*)]
    \item The regions are not disjoint. $\mathcal{R}_{cross}$ intersects $\mathcal{R}_{mix}$ and $\mathcal{R}_{neg}$.
    \item The four regions do not cover the whole $M_2(\mathbb{R})^2$. Informally speaking, they cover more than $\frac{3}{4}$ of the space.
    \item The regions $R_{cross},R_{neg},R_{copar}$ are all open and $R_{mix}$ is closed. All of these regions are semialgebraic sets.
    \item The boundary of $\mathcal{R}_{cross}$ consists of reducible pairs $(A,B)$. We will see later that it is closely related to the classical Cayley cubic surface.
    \item Any pair of entrywise non-negative matrices $(A,B)$ belongs to one of the four regions.
    \item To the best of our knowledge, all counterexamples to finiteness conjecture described in the literature lie in 
    $\mathcal{R}_{copar}$.
    \item In the unexplored regions (The complement of all four regions), there are non-empty open sets where generic uniqueness of SMP's fails: see \cite{BL}. Our classification shows that generically, all non-Sturmian SMPs must live in these unexplored regions.
    \item All examples of pairs in \cite{Panti,Kozyakin,MS} that have Sturmian SMPs belong to either $\mathcal{R}_{cross}$ or $\mathcal{R}_{copar}$.
\end{enumerate}
The paper is organized as follows:
\begin{enumerate}
    \item We introduce the framework of working on simultaneous conjugacy classes rather than matrices.
    \item We then prove the classification for $\mathcal{R}_{cross},\mathcal{R}_{mix},\mathcal{R}_{neg}$, treating each region separately.
    \item We then discuss the connection between the joint spectral radius and ergodic optimization. Using this connection we prove the theorem for the region $\mathcal{R}_{copar}$.
    \item Finally we prove the generic uniqueness statement in the main theorem, as well as classify SMPs for non-negative matrices.
\end{enumerate}

At this point we would like to make a distinction between products and words. Let $\{0,1\}^*$ denote the set of all finite sequences of $0$'s and $1$'s. Any element $w \in \{0,1\}^*$ will be referred to as a \emph{word}. Let $A,B \in M_2(\mathbb{R})$, we denote $w(A,B)$ as the matrix \emph{product} obtained by replacing the letter $0$ with matrix $A$ and $1$ with matrix $B$ in the word $w$. Technically, the statement in \cref{intro thm} should be made in terms of words.
\section{Simultaneous Conjugacy classes}
It is easy to see (using \cref{e.BW} for instance) that the $\JSR$ is invariant under simultaneous conjugation. That is why it sometimes useful to work with simultaneous conjugacy classes rather than matrices themselves. We define for a pair of $2\times 2$ complex matrices $(A,B)$ the following numbers:
\begin{equation}\label{e.5_var}
x = \tr A, \quad
y = \tr B, \quad
z = \tr AB, \quad
u = \det A,\quad 
v = \det B.
\end{equation}
We may think of $(A,B) \mapsto (x,y,z,u,v)$ as a mapping from the set of simultaneous conjugacy classes of pairs of complex $2\times2$ matrices to $\mathbb{C}^5$, as all $5$ quantities are invariant under simultaneous conjugation. We will refer to $(x,y,z,u,v)$ as the associated 5-tuples. 

Given a $5$-tuple $(x,y,z,u,v)$ it is easy to find a pair $(A,B)$ with this $5$-tuple, showing that $\phi$ is a surjection. Let $\lambda_1,\lambda_2$ and $\mu_1,\mu_2$ be the roots of $t^2-xt+u$ and $t^2-yt+v$ respectively. Consider the matrices:
\begin{equation}\label{e.surjective}
\centering
A=\begin{bmatrix}
\lambda_1 & 1\\
0 & \lambda_2
\end{bmatrix}, B=\begin{bmatrix}
\mu_1 & 0\\
z-\lambda_1\mu_1-\lambda_2\mu_2 & \mu_2
\end{bmatrix}.
\end{equation}
Using Vieta's formulas and direct computation, this pair's $5$-tuple is $(x,y,z,u,v)$ as desired.

The next result is a theorem by Friedland that provides information on the injectivity of the map $\phi$. Results in this direction restricted to $SL(2,\mathbb{R})$ existed in literature much earlier, see \cite{Goldman} for a modern overview and \cite{Vogt,Fricke} for original references.
\begin{theorem}\cite{Friedland}\label{friedland}
The simultaneous conjugacy class of $(A,B) \in M_2(\mathbb{C})^2$ is uniquely determined by its $5$-tuple $(x,y,z,u,v)$ as long as
\[4 u v-u y^2-v x^2+x y z-z^2\not= 0.\]
Otherwise, the $5$-tuple does not uniquely determine the conjugacy class and all the pairs of matrices associated to it are reducible.
\end{theorem}
There are many useful alternative formulations of the equation mentioned in \cref{friedland}, we include them without proof since these can be verified using direct computations.
\begin{theorem}\label{char. of red.}
Let $(A,B) \in M_2(\mathbb{C})^2$ and let $(x,y,z,u,v)$ be the associated $5$-tuple. Then, the following quantities are equal:
\end{theorem}
\begin{enumerate}
 \item $4 u v-u y^2-v x^2+x y z-z^2$
 \item $\det(AB-BA)$
 \item $\frac{1}{4}(\tr(A)^2-4\det{A})(\tr(B)^2-4\det{B})-(\tr(AB)-\frac{1}{2}\tr(A)\tr(B))^2$
 \item $\tr(A^2B^2)- \tr((AB)^2)$
 \item $\det(A)\det(B)(2-\tr(ABA^{-1}B^{-1}))$ (if the matrices are invertible)
\end{enumerate}
\cref{friedland} is stated for complex pairs of $2\times 2$ matrices. In this paper we will restrict ourselves to real matrices only. It is not hard to see which $5$-tuples come from a conjugacy class with real representatives using the following theorem (See also: \cite{Morgan}):
\begin{theorem}\cite[Proposition 4.1]{BL}\label{real-realizable}
A real $5$-tuple $(x,y,z,u,v)\in \mathbb{R}^5$ can be attained by a pair of real matrices $(A,B)$ if and only if \begin{equation}
\min \big( 4u-x^2,4 u v-u y^2-v x^2+x y z-z^2\big) \le 0 \, .
\end{equation}
\end{theorem}
We should also note that \cref{friedland}  is sufficient to prove that the joint spectral radius of $(A,B)$ only depends on the associated $5$-tuple.
\begin{theorem}
There exists a continuous function $J:\mathbb{C}^5\to \mathbb{R}$ such that $J(x,y,z,u,v)=\JSR(A,B)$ for any $(A,B)$ whose associated $5$-tuple is $(x,y,z,u,v)$.
\end{theorem}
\begin{proof}
We first establish the existence of such function $J$. If $4 u v-u y^2-v x^2+x y z-z^2= 0$ then by \cref{friedland} any $(A,B)$ that maps to the $5$-tuple must be reducible, hence
\[
 J(x,y,z,u,v)=\max\{\rho(A),\rho(B)\}=\max\left\{\left|\frac{x\pm\sqrt{x^2-4u}}{2}\right|,\left|\frac{y\pm\sqrt{y^2-4v}}{2}\right|\right\}.
\]
 
 On the other hand, if $4 u v-u y^2-v x^2+x y z-z^2\not = 0$, then the $5$-tuple corresponds to a unique conjugacy class. Therefore, the $\JSR$ becomes uniquely determined. This shows that the function $J$ exists and is unique.
 
 Denote by $K(\mathbb{C}^8)$ the space of all compact subsets of $\mathbb{C}^8$, equipped with the Hausdorff metric. Define a mapping $\psi : \mathbb{C}^5 \to K(\mathbb{C}^8)$
by associating to each 5-tuple the set of all possible pairs $(A,B)$ of complex matrices of the form in \cref{e.surjective}. Note that there are at most four such pairs, corresponding to permutations of the roots $\lambda_1,\lambda_2$ and $\mu_1,\mu_2$. Since the zero set of a monic polynomial depends continuously on its coefficients, the mapping \(\psi\) is continuous.

Moreover, since \(\JSR\) is continuous (see, e.g., \cite[Proposition 1.10]{Jungers} or \cite{Jeremias}), the function
\[
M : K(\mathbb{C}^8) \to \mathbb{R}, \quad M(K) = \max_{(A,B) \in K} \JSR(A,B)
\]
is continuous. Consequently, the continuity of \(J\) follows from the commutativity of the diagram below.

\begin{figure}[h]\label{diagram}
\scalebox{.85}{
\centering
\begin{tikzpicture}[node distance=2cm, auto]
\node (A) {$K(\mathbb{C}^8)$};
\node(B) [right of=A] {$\mathbb{R}$};
\node (C) [below of=A] {$\mathbb{C}^5$};
\draw[->](A) to node {$M$}(B);
\draw[<-](A) to node [left] {$\psi$}(C);
\draw[->](C) to node [below=0.5ex] {$J$}(B);
\end{tikzpicture}}
\end{figure}

\end{proof}
We will also need the following theorem originally due to Fricke:
\begin{theorem}\label{e. fricke}
Let $w \in \{0,1\}^*$, then there exists a unique polynomial $F_w(x,y,z,u,v)$ with integer coefficients such that for any $A,B \in M_2(\mathbb{R})$ the following holds:
\begin{equation}
\tr{w(A,B)}=F_w(x,y,z,u,v).
\end{equation}
\end{theorem}
The polynomial $F_w(x,y,z,u,v)$ is called the \emph{Fricke} polynomial. We omit the proof, for details see \cite{Horowitz,Magnus,Goldman}.
\section{Geometric description of a few regions of interest}
A $GL(2,\mathbb{C})$ matrix is naturally associated to a M\"{o}bius transformation of the Riemann sphere $\hat{\mathbb{C}}$ by the following mapping
\[
A=\begin{bmatrix}
    a & b\\
    c & d
\end{bmatrix}\mapsto 
    f_A(\zeta):=\frac{a\zeta+b}{c\zeta+d}.
\]
We list a few standard properties of $f_A$ and how they relate to $A$.
\begin{itemize}
    \item If $A \in GL(2,\mathbb{R})$, then $f_A$ preserves the real circle $\hat{\mathbb{R}}=\mathbb{R}\cup \{\infty\}$.
    \item If $A \in GL^+(2,\mathbb{R})$, then $f_A$ preserves the upper half-plane $H=\{\zeta \in \mathbb{C}:\text{Im}(\zeta)>0\}$ and ${f_A|}_{H}$ is an isometry with respect to the hyperbolic metric.
    \item If $(w_1,w_2)\in \mathbb{C}^2$ is an eigenvector of $A$, then $\frac{w_1}{w_2}$ is a fixed point of $f_A$.
    \item If $A \in GL(2,\mathbb{R})$ is real diagonalizable with distinct eigenvalues, then $f_A$ has exactly two fixed points in $\hat{\mathbb{C}}$, both contained in $\hat{\mathbb{R}}$, one being an attractor and the other being a repellor under iteration of $f_A$. If in addition, $\det(A)>0$, then there is a unique hyperbolic geodesic which is invariant under $f_A$ called the \emph{translation axis} of $A$. We orient the translation axis so that it points towards the attractor of $f_A$.
\end{itemize}
Equipped with these properties we may now introduce crucial definitions.
\begin{definition}
Let $(A,B)$  be an irreducible pair of real diagonalizable matrices. Let $x_A,y_A$ (respectively $x_B,y_B$) be the fixed points of $f_A$ (resp. of $f_B$) in the circle $\hat{\mathbb{R}}$. We say the pair $(A,B)$ is crossing if the points $x_A$ and $y_A$ belong to different connected components of $\hat{\mathbb{R}}\setminus\{x_B,y_B\}$.
\end{definition}
\begin{note}
In the case of positive determinants, if $(A,B)$ is crossing, then the two translation axes actually cross (intersect), which justifies the terminology.
\end{note}
\begin{definition}
Let $(A,B) \in GL^+(2,\mathbb{R})^2$ be an irreducible pair of real diagonalizable matrices.
\end{definition}
\begin{itemize}
    \item If the translation axes of $A$ and $B$ do not intersect and induce opposite orientation on the region they bound (The region between the respective translation axis and the real line), we say $(A,B)$ is \emph{co-parallel}.
    \item If the translation axes of $A$ and $B$ do not intersect and induce the same orientation on the region they bound, we say $(A,B)$ is \emph{anti-parallel}.
\end{itemize}
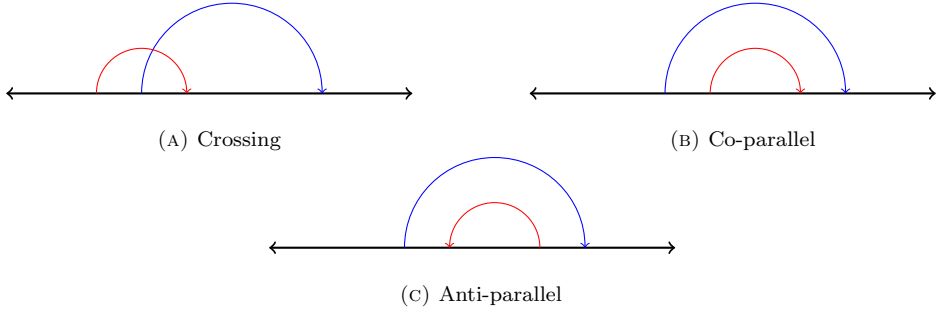
\begin{figure}[h]

\begin{subfigure}{0.45\textwidth}
\begin{tikzpicture}[scale=0.6]
\draw[thick,<->] (-4.5,0) -- (4.5,0) node[anchor=north west] {};
\draw[blue,<-] (2.5,0) arc (0:180:2cm);
\draw[red,<-] (-0.5,0) arc (0:180:1cm);
\end{tikzpicture}
\caption{Crossing}
\end{subfigure}
\hfill
\begin{subfigure}{0.45\textwidth}
\begin{tikzpicture}[scale=0.6]
\draw[thick,<->] (-4.5,0) -- (4.5,0) node[anchor=north west] {};
\draw[blue,<-] (2.5,0) arc (0:180:2cm);
\draw[red,<-] (1.5,0) arc (0:180:1cm);
\end{tikzpicture}
\caption{Co-parallel}
\end{subfigure}
\hfill
\begin{subfigure}{0.45\textwidth}
\centering
\begin{tikzpicture}[scale=0.6]
\draw[thick,<->] (-4.5,0) -- (4.5,0) node[anchor=north west] {};
\draw[blue,<-] (2.5,0) arc (0:180:2cm);
\draw[red,->] (1.5,0) arc (0:180:1cm);
\end{tikzpicture}
\caption{Anti-parallel}
\end{subfigure}
\caption{Example configurations of the translation axes.}
\end{figure}
These definitions are the starting points for formally defining 3 regions of interest.
\begin{align*}
(A,B) \in \mathcal{R}_{cross} & \iff 
(A,B) \text{ is crossing}
\\
(A,B) \in \mathcal{R}_{copar} & \iff 
(A,B) \text{ is co-parallel}
\\
(A,B) \in \mathcal{S}_{anti} & \iff 
(A,B) \text{ is anti-parallel}
\end{align*}
Notice that if $(A,B)$ belongs to one of the three regions, then implicitly $(A,B)$ is irreducible and both matrices must be real diagonalizable. If $(A,B)$ is anti-parallel or co-parallel, we are also implicitly assuming both $A,B \in GL^+(2,\mathbb{R})$.
\begin{note}
When $A,B$ are simultaneously conjugated, independently re-scaled or swapped, the order of the fixed points $\{x_A,y_A,x_B,y_B\}$ is either unchanged or it is reversed. Therefore, being crossing/co-parallel/anti-parallel is invariant under these operations.
\end{note}
These are geometric definitions; however, it will be useful to have an algebraic description of these conditions on hand. These will give us a simple way to determine which configuration a given pair is. It will also formally justify the descriptions of the regions given in the introduction.

\section{Algebraic description of regions $\mathcal{R}_{cross}$, $\mathcal{R}_{copar}$ and $\mathcal{S}_{anti}$}
We begin by algebraically describing the crossing configuration. This was originally done by Jorgensen-Smith \cite{JorgensenSmith}, but is stated in slightly lower generality than we shall need. Thus, we reproduce the original proof which carries over to the generality needed.
\begin{theorem}\label{crossing}
Let $A,B$ be real diagonalizable, then the following are equivalent:
\begin{itemize}
    \item $(A,B)$ is crossing, in other words $(A,B) \in \mathcal{R}_{cross}$ 
 \item $\det(AB-BA)>0$
\end{itemize}
 
\end{theorem}
\begin{proof}
 By \cref{friedland}, \cref{char. of red.} and the fact that crossing matrices must be irreducible, we may assume $(A,B)$ is irreducible. Since the pair is irreducible and real diagonalizable, by appropriately conjugating we may assume that
\[A=\begin{bmatrix}
a_1&0\\
0&a_2
\end{bmatrix} \text{  and  } B=\begin{bmatrix}
b_1&b_2\\
b_3&b_4
\end{bmatrix}\]
where  $b_3\not=0$, $b_2\not=0$ and $a_1\not=a_2$.

Since $A$ is diagonal, $f_A$ has fixed points $\{0,\infty\}$. Thus, $(A,B)$ is crossing if and only if $f_B$ has a positive and a negative fixed point. Hence, we need to find conditions for the product of the fixed points of $f_B$ to be negative. However, fixed points of $f_B$ are the zeroes of the following polynomial (where $z$ is just a variable in this case):
\[b_3z^2+(b_4-b_1)z-b_2.\]
By Vieta's formulas we know that the product of the fixed points is exactly $-\frac{b_2}{b_3}$. This is enough to prove the claim. By direct computation
\[\det(AB-BA)=b_2b_3(a_1-a_2)^2,\]
and so:
\[-\frac{b_2}{b_3}<0\iff \det(AB-BA)>0.\]
\end{proof}
\begin{note}
By \cref{char. of red.,crossing} the real diagonalizable matrices on the boundary of the crossing region $\mathcal{R}_{cross}$ satisfy:
\[\det(AB-BA)=4 u v-u y^2-v x^2+x y z-z^2=0.\]
By \cref{friedland} these must be reducible pairs $(A,B)$. Setting $u=v=1$ we get the  equation of the classical Cayley cubic surface.
\end{note}

Next we provide an algebraic description of being co-parallel. We will omit the proof, this was also done by Jorgensen-Smith.
\begin{theorem}\cite{JorgensenSmith}\label{e. jorgensensmith}
Let $(A,B) \in SL(2,\mathbb{R})^2$ be irreducible, real diagonalizable and with positive traces, then $(A,B)$ is co-parallel (in other words $(A,B) \in \mathcal{R}_{copar}$) if and only if $(A,B)$ is not crossing and $\tr{AB}>\frac{1}{2}\tr{A}\tr{B}$.
\end{theorem}
Observe that due the independent re-scaling invariance, for any non-zero $\alpha,\beta$:
\begin{equation}\label{invariance}
(A,B) \in \mathcal{R}_{cross}/\mathcal{R}_{copar}/\mathcal{S}_{anti} \iff \left(\alpha {A},\beta B\right ) \in \mathcal{R}_{cross}/\mathcal{R}_{copar}/\mathcal{S}_{anti}.
\end{equation}
Using this, it is not hard to see that the description of the co-parallel region given by the above theorem can be extended to the description given in the introduction (Right before \cref{intro thm}). Furthermore, it shows no generality is lost when assuming $A,B$ have non-negative trace in the following theorem.
\begin{theorem}\label{algebraic desc regions}
Let $(A,B) \in \mathcal{D}^2$ with $\tr{A},\tr{B}\geq 0$. Denote $(x,y,z,u,v)$ as the associated $5$-tuple, then:
\begin{align*}
(A,B) \in \mathcal{R}_{cross} & \iff 
    \frac{x y}{2}-\frac{1}{2} \sqrt{\left(x^2-4 u\right) \left(y^2-4 v\right)}< z< \frac{x y}{2}+\frac{1}{2}
   \sqrt{\left(x^2-4 u\right) \left(y^2-4 v\right)}
\\
(A,B) \in \mathcal{R}_{copar} & \iff 
z> \frac{x y}{2}+\frac{1}{2} \sqrt{\left(x^2-4u\right) \left(y^2-4v\right)} \text{  and  } u,v>0
\\
(A,B) \in \mathcal{S}_{anti} & \iff 
z<\frac{x y}{2}-\frac{1}{2} \sqrt{\left(x^2-4u\right) \left(y^2-4v\right)} \text{  and  } u,v>0
\end{align*}
\end{theorem}
\begin{proof}
By \cref{crossing}:
\[(A,B) \in \mathcal{R}_{cross} \iff \det(AB-BA)>0 \]
Using the first characterization in \cref{char. of red.} and applying the quadratic formula on $z$ we may re-write the inequality as\begin{equation}\label{crossing ineq}
    \frac{x y}{2}-\frac{1}{2} \sqrt{\left(x^2-4 u\right) \left(y^2-4 v\right)}< z< \frac{x y}{2}+\frac{1}{2}
   \sqrt{\left(x^2-4 u\right) \left(y^2-4 v\right)},
\end{equation}
which proves the first part.

If $(A,B)$ is co-parallel, the matrices must have positive determinants by definition, thus $u,v>0$. Now by since $(A,B)$ is not crossing,  \cref{crossing ineq} does not hold. That is $z$ is either:
\[z\leq \frac{x y}{2}-\frac{1}{2} \sqrt{\left(x^2-4u\right) \left(y^2-4v\right)} \text{ 
  or  }z\geq  \frac{x y}{2}+\frac{1}{2} \sqrt{\left(x^2-4u\right) \left(y^2-4v\right)}\]
By \cref{e. jorgensensmith} only the second alternative holds.
The inequality is strict, otherwise $(A,B)$ would be reducible by \cref{friedland}.

If $(A,B)$ is anti-parallel the matrices must have positive determinants by definition, thus $u,v>0$. Since we characterized crossing and co-parallel matrices, the anti-parallel region falls out by just looking at what is left:
\[z<\frac{x y}{2}-\frac{1}{2} \sqrt{\left(x^2-4u\right) \left(y^2-4v\right)}\]
Again the inequality must be strict, otherwise $(A,B)$ would be reducible.
\end{proof}
A useful algebraic consequence of the above that will be used later is:
\begin{corollary}\label{alt}
If $(A,B)$ is co-parallel then:  
\begin{equation}
    \rho(AB)>\rho(A)\rho(B).
\end{equation}
\end{corollary}
\begin{proof}
The inequality to be proven is invariant under re-scaling of the matrices. Therefore, by \cref{invariance} we may assume $A,B\in SL(2,\mathbb{R})$ and that $\tr{A},\tr{B}\geq0$. In fact, since $(A,B)$ is an irreducible pair of real diagonalizable matrices, $\tr{A},\tr{B}> 2$. By \cref{e. jorgensensmith}, $\tr{AB}>2$. Therefore, we have the following formulas:
\[\rho(A)=\frac{x+\sqrt{x^2-4}}{2},\quad \rho(B)=\frac{y+\sqrt{y^2-4}}{2}, \quad \rho(AB)=\frac{z+\sqrt{z^2-4}}{2} \]
Using these, we may re-write the co-parallel case from \cref{algebraic desc regions} in the following way:
\[\rho(AB)+\frac{1}{\rho(AB)}>\rho(A)\rho(B)+\frac{1}{\rho(A)\rho(B)}\]
Since $\rho(AB)>1$ and $\rho(A)\rho(B)>1$, it follows that
\[\rho(AB)>\rho(A)\rho(B).\]
\end{proof}
\section{What are the possible SMPs when $(A,B) \in \mathcal{R}_{cross}$?}
In this section, we will describe all possible SMPs given that $(A,B)$ is crossing. From here on, we denote $W(a,b)$ as the set of products that have $a$ factors of $A$ and $b$ factors of $B$. We begin with the following lemma:
\begin{lemma}\label{strict submult}
If $A,B \in M_2(\mathbb{R})$ are an irreducible pair of symmetric matrices then for any $\Pi \in W(a,b)$ with $a,b\geq 1$:
\[\|\Pi\|_2<\|A\|_2^a\|B\|_2^b\]
\end{lemma}
\begin{proof}
First note that
\[\|\Pi\|_2\leq\|AB\|_2\|A\|_2^{a-1}\|B\|_2^{b-1}\]
by sub-multiplicativity. That is why it is sufficient to prove that $\|AB\|_2=\|A\|_2\|B\|_2$ is impossible. Assume the equality holds, then the singular vector of $B$ must map to a singular vector of $A$. Since $A,B$ are symmetric matrices, the singular vectors are eigenvectors and consequently, an eigenvector of $B$ must map into an eigenvector of $A$. However, that would mean $A$ and $B$ share an eigenvector and so $A,B$ are reducible, which is a contradiction.
\end{proof}
\begin{theorem}\label{symmetrizing}
$(A,B)$ is crossing if and only if $(A,B)$ is irreducible and the matrices are simultaneously conjugate to symmetric matrices.
\end{theorem}
\begin{proof}
We first prove the forward direction. We may assume $A$ is a diagonal matrix by appropriately conjugating both matrices. Thus, the eigendirections of $A$, $L_1,L_3$ are simply the $x$ and $y$ axis respectively. Denote $L_2,L_4$ as the eigendirections of $B$. Then by assumption of crossing we know that these lines are cyclically ordered $L_1,L_2,L_3,L_4$. Now consider conjugating both matrices $A,B$ by
$D=\begin{bmatrix}
\frac{1}{\lambda} & 0\\
0& \lambda 
\end{bmatrix}$, it leaves $L_1,L_3$ invariant while rotating the lines $L_2,L_4$. Now the angle between $L_2$ and $L_4$, denoted $\angle(D(L_2),D(L_4))$, changes continuously with $\lambda$. If $\lambda\to 0$ then $\angle(D(L_2),D(L_4))\to 0$  and if $\lambda\to \infty$ then $\angle(D(L_2),D(L_4))\to \pi$. Thus by intermediate value theorem there exists $\lambda$ for which the angle is $\frac{\pi}{2}$. Therefore by conjugating using this $\lambda$, $DAD^{-1}$ is still diagonal while $DBD^{-1}$ has orthogonal eigendirections. Thus, both are symmetric.

For the backwards directions, we assume $(A,B)$ is an irreducible pair of symmetric matrices. Thus the eigendirections of $A$ (resp. $B$) are orthogonal. The only way $(A,B)$ is not crossing, is if $A$ and $B$ have the same eigendirections, but that would mean they are reducible. Thus $(A,B)$ must be crossing.
\end{proof}
\begin{note}
It is not hard to explicitly exhibit the symmetric matrices. If $(A,B)$ is crossing and has the $5$-tuple $(x,y,z,u,v)$,  then the pair is simultaneously conjugate to:
 \begin{align}
A^{\text{sym}}&=
\begin{bmatrix}
 \frac{1}{2} \left(x+\sqrt{x^2-4 u}\right) & 0 \\
 0 & \frac{1}{2} \left(x-\sqrt{x^2-4 u}\right) \\
\end{bmatrix}\\
B^{\text{sym}}&=
\begin{bmatrix}
 \frac{y \sqrt{x    ^2-4 u}-x y+2 z}{2 \sqrt{x^2-4 u}} & \sqrt{\frac{4 u v-u y^2-v x^2+x y z-z^2}{x^2-4 u}} \\
\sqrt{\frac{4 u v-u y^2-v x^2+x y z-z^2}{x^2-4 u}} & \frac{y \sqrt{x^2-4
   u}+xy-2 z}{2 \sqrt{x^2-4 u}} \\
\end{bmatrix}
\end{align}
\end{note}
From this, we obtain the first part of \cref{intro thm}:
\begin{corollary}\label{crossing smps}
If $(A,B) \in \mathcal{R}_{cross}$, then an SMP exists and all SMPs are either $A$ or $B$.
\end{corollary}
\begin{proof}
 The joint spectral radius and SMPs are invariant under simultaneous conjugation. Thus, by \cref{symmetrizing} we can assume $A,B$ are both symmetric. By \cref{three member} for $k=1$ and Euclidean norm $\|\cdot\|_2$ we know that:
\[\max\{\rho(A),\rho(B)\}\leq \JSR(A,B)\leq \max(\|A\|_2,\|B\|_2)\]
However, since $A,B$ are symmetric  $\rho(A)=\|A\|_2, \rho(B)=\|B\|_2$. Therefore the inequalities become equalities:
\[\max\{\rho(A),\rho(B)\}=\JSR(A,B)=\max(\|A\|_2,\|B\|_2)\]
This proves that either $A$ or $B$ is an SMP.

We now prove there can be no other SMPs. Let $\Pi \in W(a,b)$ be a primitive product with $a,b \geq 1$. Then we get the following chain of inequalities:
\[\rho(\Pi)\leq \|\Pi\|_2<\|A\|_2^a\|B\|_2^b=\rho(A)^a\rho(B)^b\leq \max\{\rho(A)^{a+b},\rho(B)^{a+b}\}\]
The second inequality must be strict by \cref{strict submult}. Thus $\Pi$ cannot be an SMP.
\end{proof}
\begin{note}
This corollary is not completely new. Panti and Sclosa \cite{Panti} proved a related result through different methods. 
\end{note}

\section{What are the SMPs if $(A,B) \in \mathcal{R}_{mix}$?}
The region $\mathcal{R}_{mix}$ is the region where the matrices $A,B$ have opposite signs of the determinants or if either matrix is non-invertible. We remind the reader that if a real matrix has a negative determinant, it must be real-diagonalizable. This fact will be implicitly used throughout the next sections. We first must deal with the case when $\JSR(A,B)=0$. (The lemma below remains true in any dimension, see \cite[Lemma 2.2]{Jungers}),
\begin{lemma}\label{jsr zero}
Let $A,B \in M_2(\mathbb{R})$ such that $\JSR(A,B)=0$, then $A,B$ are reducible.
\end{lemma}
\begin{proof}
Since $\JSR(A,B)=0$, we know by \cref{three member} that:
\[\rho(A)=\rho(B)=\rho(AB)=0\]
However, that means all the eigenvalues of $A,B,AB$ are $0$ and so:
\[\tr(A)=\tr(B)=\tr(AB)=\det(A)=\det(B)=0.\]
Hence $(A,B)$ must be reducible by \cref{friedland}.
\end{proof}
The above lemma will allow us to always re-scale our matrices $A,B$ so that $\JSR(A,B)=1$. We now prove a few auxiliary lemmas that will be used in the main theorems. 
\begin{lemma}\label{cool}
For any $2\times 2$ matrices $A,B$ the following equation holds
\[\det{(A+B)}+\tr{AB}=\det{A}+\det{B}+\tr{A}\tr{B}\]
\end{lemma}
\begin{proof}
Taking the trace of the Cayley Hamilton equation we obtain the identity:
\[\tr{(A^2)}=\tr({A})^2-2\det{A}\]
Applying the same identity to the matrices $B$ and $A+B$, using linearity of trace and some algebra, we obtain the desired formula.
\end{proof}
\begin{corollary}\label{multiplicative trace}
Let $X,Y,Z \in M_2(\mathbb{R})$ and $\det(Z)=0$, then
\[\tr{(XZYZ)}=\tr(XZ)\tr(YZ)\]
\end{corollary}
\begin{proof}
Setting $A=XZ$ and $B=YZ$ in \cref{cool} we get:
\[\det{(XZ+YZ)}+\tr{(XZYZ)}=\det{XZ}+\det{YZ}+\tr{(XZ)}\tr{(YZ)}\]
However, since $\det(Z)=0$, we get the desired equality.
\end{proof}
\begin{lemma}\label{seq}
Let $a_n,b_n$ be positive sequences such that 
\[a_n\leq b_n  \text{  and  } b_n \to b\]
Then $\sup_{n\geq 1}\{a_n,b\}$ is attained.
\end{lemma}
\begin{proof}
Since $(b_n)$ is is bounded, so is $(a_n)$. If $a:=\sup_{n\geq 1}\{a_n\}$ is attained the lemma is obvious. If it is not attained, then there exists a sub-sequence such that $a_{n_i} \to a$. Since $a_n\leq b_n$, we get $a\leq b$. In this case, $b$ would attain the supremum.
\end{proof}
\begin{lemma}\label{gelfandv2}
For any $n\in \mathbb{N}$, if $A,B \in M_n(\mathbb{C})$, the supremum:
\[\sup_{n \geq 0}\left\{\rho(A^nB)^{\frac{1}{n+1}},\rho(A)\right\}\]
is attained.
\end{lemma}
\begin{proof}
Note the following inequalities:
\begin{equation}
    0\leq \rho(A^nB)^\frac{1}{n+1}\leq \|A^nB\|^\frac{1}{n+1}\leq \|A^n\|^\frac{1}{n+1}\|B\|^\frac{1}{n+1}.
\end{equation}
By Gelfand's formula (as long as $B$ is not the zero matrix, which would make the lemma trivial):
\[\|A^n\|^\frac{1}{n+1}\|B\|^\frac{1}{n+1} \to \rho(A).\]
Thus setting $a_n=\rho(A^nB)^\frac{1}{n+1}, b_n=\|A^n\|^\frac{1}{n+1}\|B\|^\frac{1}{n+1}, b=\rho(A)$ and applying \cref{seq} we see that the supremum is attained.
\end{proof}
We now begin the proof of the 2\textsuperscript{nd} part of \cref{intro thm}, starting with the case where one of the matrices is not invertible. We note that the following lemma overlaps with \cite[Theorem 2.1]{dai}. In particular, our lemma is a special case of their result, but with a slightly more precise conclusion.
\begin{lemma}\label{non-inv}
Let $(A,B)$ be irreducible such that $\det(B)=0$, then there is an SMP of the form $A,B,A^nB$ where $n \in \mathbb{N}$.
\end{lemma}
\begin{proof}
By \cref{jsr zero} we may normalize $A,B$ so that $\JSR(A,B)=1$. Any product (up to cyclic permutations) other than powers of $A$ can be written as 
\[\Pi=A^{k_1'}B\cdots A^{k_n'}B \]
with $k_i'\geq0$.
By Cayley-Hamilton $B^n=(\tr{B})^{n-1}B$. If $k_{i}'=0$ for some $i$, then $\Pi$ has a power of $B$  which can be reduced by pulling out $\tr{B}$.  Therefore, $\Pi$ can be written as
\[\Pi=(\tr{B})^k(A^{k_1}B\cdots A^{k_n}B)\]
with $k_{i}\geq 1$. We now see that
\[\rho(A^{k_1}B\cdots A^{k_n}B)=|\tr(A^{k_1}B\cdots A^{k_n}B)|=|\tr(A^{k_1}B)|\cdots |\tr(A^{k_n}B)|\]
where the second equality holds by repeatedly applying \cref{multiplicative trace}.
Therefore
\[\rho(\Pi)=\rho((\tr{B})^k(A^{k_1}B\cdots A^{k_n}B)))=\rho(B)^{k} \rho(A^{k_1}B)\cdots \rho(A^{k_n}B)\]

Now consider 
\[M=\sup_{n \geq 0}\left\{\rho(A^nB)^{\frac{1}{n+1}},\rho(A)\right\}\]
By \cref{gelfandv2} we know that the supremum is attained. Assume towards contradiction that there is no SMP of the form $A,B,A^nB$. Then we must have $M<1$. However, that would mean by decomposition above that for any $\Pi$, 
\[\rho(\Pi)^\frac{1}{|\Pi|}\leq M<1\] but that contradicts \cref{e.BW}. Thus there must be an SMP among $A,B,A^nB$.
\end{proof}
\begin{lemma}\label{lemma cross}
If $(A,B)$ is an irreducible pair of invertible matrices such that $AB$ is real diagonalizable and $\det(B)<0$, then either $(A,B)$ is crossing or $(B,AB)$ is crossing.
\end{lemma}
\begin{proof}
We show that if $(A,B)$ is not crossing, then $(B,AB)$ must be crossing. We have two cases, either $A$ is or is not real diagonalizable. 

If $A$ is real diagonalizable and $(A,B)$ is not crossing then by \cref{crossing} and irreducibility of $(A,B)$ we know that:
\[\det(AB-BA)< 0\]
Now applying the same theorem to the pair $(B,AB)$ we need to determine the sign of the following expression
\[\det(BAB-ABB)=\det(BA-AB)\det(B)=\det(AB-BA)\det(B)>0\]
Which is positive since $\det(B)<0$ and so the sign is flipped. Hence $(B,AB)$ is crossing.

If $A$ is not real diagonalizable, then necessarily we have $\tr(A)^2-4\det(A)\leq 0$ and $\det(B)<0$ by assumption. Under these conditions, it is easy to see using the third characterization in \cref{char. of red.} and irreducibility of $(A,B)$ that
\[\det(AB-BA)< 0.\]
But then $(B,AB)$ is crossing by repeating the argument in the previous case.

\end{proof}
\begin{lemma}\label{opposite determinants}
If $(A,B)$ is an irreducible pair of invertible matrices such that $(B,AB)$ is crossing, then the pair $\{A,B\}$ must have an SMP of the form $A,B,A^nB$. Furthermore, every SMP is of this form (up to cyclic permutation) unless there exists $m\geq 2$ such that 
\[A^m=\JSR(A,B)^mI.\]
\end{lemma}
\begin{proof}
First we normalize $A,B$ so that $\JSR(A,B)=1$, we can do so by \cref{jsr zero}.
 Since $(B,AB)$ is crossing,  by \cref{symmetrizing}, $(B,AB)$ can be simultaneously symmetrized. Thus, we may assume $B,AB$ are symmetric. If that is the case, then any matrix of the form $A^nB$ is also symmetric; indeed:
\[(A^nB)^T=((ABB^{-1})^{n-1}AB)^T=AB((ABB^{-1})^T)^{n-1}=AB(B^{-1}AB)^{n-1}=A^nB\]
Now we essentially repeat the arguments in \cref{non-inv}. Consider
\[M=\sup_{n \geq 0}\left\{\rho(A^nB)^{\frac{1}{n+1}},\rho(A)\right\}\]
We can apply \cref{gelfandv2} to conclude that the supremum $M$ must be attained. We assume towards contradiction that there is no SMP of the form $A,B,A^nB$, in other words $M<1$. Any product $\Pi$ (up to cyclic permutation) other than powers of A can be written as
\[\Pi=A^{k_1}B\cdots A^{k_n}B\]
with $k_i\geq 0$.
However, since $A^{k_i}B$ are all symmetric we have $\rho(A^{k_i}B)=\|A^{k_i}B\|_2$ and therefore
 \[\rho(\Pi)\leq \|\Pi\|_2\leq\|A^{k_1}B\|_2\cdots\|A^{k_n}B\|_2=\rho(A^{k_1}B)\cdots \rho(A^{k_n}B)\]
 However, due to the above we have 
 \[\rho(\Pi)^\frac{1}{|\Pi|}\leq M<1\]
 which would contradict \cref{e.BW}. Hence $M=1$ and there is an SMP of the desired form.
 
We now prove there can be no SMP of any other form unless there exists $m\geq 2$ such that $A^m=I$. Assume $\Pi$ is an SMP that is not of the desired form, then by the same decomposition 
 \[\rho(\Pi)\leq \|\Pi\|_2\leq\|A^{k_1}B\|_2\cdots\|A^{k_n}B\|_2=\rho(A^{k_1}B)\cdots \rho(A^{k_n}B)\]
Due to the above, for $\Pi$ to be an SMP, it must be the case that for each $i$, $A^{k_i}B$ is an SMP. Now since $\Pi$ is primitive and we are allowed to permute it cyclically, we may assume there exists $n>k$ such that $A^nBA^kB$ is a factor of $\Pi$. Hence, we must have
\[\|A^{n}BA^{k}B\|= \|A^nB\|\|A^kB\|, \]
otherwise $\Pi$ would not be an SMP. Since $A^nB,A^kB$ are both symmetric we can apply \cref{strict submult} to conclude $(A^nB,A^kB)$ is reducible. However, by \cref{char. of red.} (2):
\[0=\det(A^nBA^kB-A^kBA^nB)=\det(A)^{2k}\det(B)\det(A^{n-k}B-BA^{n-k})\]
which implies that $(A^{n-k},B)$ is reducible since $A,B$ are invertible. This means that either $(A,B)$ is reducible which is a contradiction or that $A^{n-k}$ is a multiple of the identity. However, if it is a multiple of the identity then, since both $A^nB,A^kB$ are SMPs and $A^nB=A^{n-k}A^kB$, we must have $A^{n-k}=I$. This proves the claim as $n-k\geq 2$, otherwise $A=I$ and $(A,B)$ would be reducible.
\end{proof}
We are now ready to prove the 2\textsuperscript{nd} part of \cref{intro thm} now for invertible matrices:
\begin{theorem}\label{ mixed case}
 If the pair $(A,B)$ is irreducible with $\det(A)>0, \det(B)<0$ and there is no $k\geq 2$ such that $A^k=\JSR(A,B)^kI$, then all the SMPs must be of the form $A,B,A^nB$.
\end{theorem}
\begin{proof}\label{opposite determinant}
By \cref{lemma cross} we have two cases either $(A,B)$ or $(B,AB)$ is crossing. If $(A,B)$ is crossing, we are done as any SMP must be either $A$ or $B$ by \cref{crossing smps}.
Alternatively, if $(B,AB)$ is crossing then by \cref{opposite determinants} all the SMP's must be of the form $A,B,A^nB$.
\end{proof}
We now show that all possible SMPs listed in \cref{ mixed case} actually are SMPs for some pair $(A,B)$ satisfying the conditions. The following construction was heavily inspired by \cite[Example 5.2]{Mejstrik}.
\begin{theorem}
For each $n \geq 1$, there exists a pair of matrices $(A_n,B_n)$ such that:
\begin{itemize}
    \item $A_n$ has real eigenvalues and it is not diagonalizable
    \item $B_n$ is non-invertible
    \item For any $(A,B)$, sufficiently close to $(A_n,B_n)$, the unique SMP is $A^nB$
\end{itemize}
\end{theorem}
Note that a pair $(A,B)$ (sufficiently close to $(A_n,B_n)$) can be chosen such that all hypothesis of \cref{ mixed case} are satisfied, showing that the word $A^nB$ is indeed a possible SMP.
\begin{proof}
Denote $c=0.278\dots$ as the unique root of the equation:
\[xe^{x+1}=1.\]
Equivalently, $c=W(1/e)$ in terms of Lambert $W$ function.
Now consider the following matrix and vectors generated by it:
\[A_n=c^\frac{1}{n}\begin{bmatrix}
1 & 0\\
1 & 1
\end{bmatrix},\quad  v_i=A_n^i\begin{bmatrix}
1\\
0
\end{bmatrix}=\begin{bmatrix}
c^\frac{i}{n} \\
ic^\frac{i}{n}
\end{bmatrix}\]
First note that for any $i \geq 0$, $v_i$ lies on the strictly concave graph of the function:
\[y=\frac{nx\log x}{\log c}
\quad (x>0).\]
\begin{figure}[h]
\centering
\includegraphics[scale=0.35]{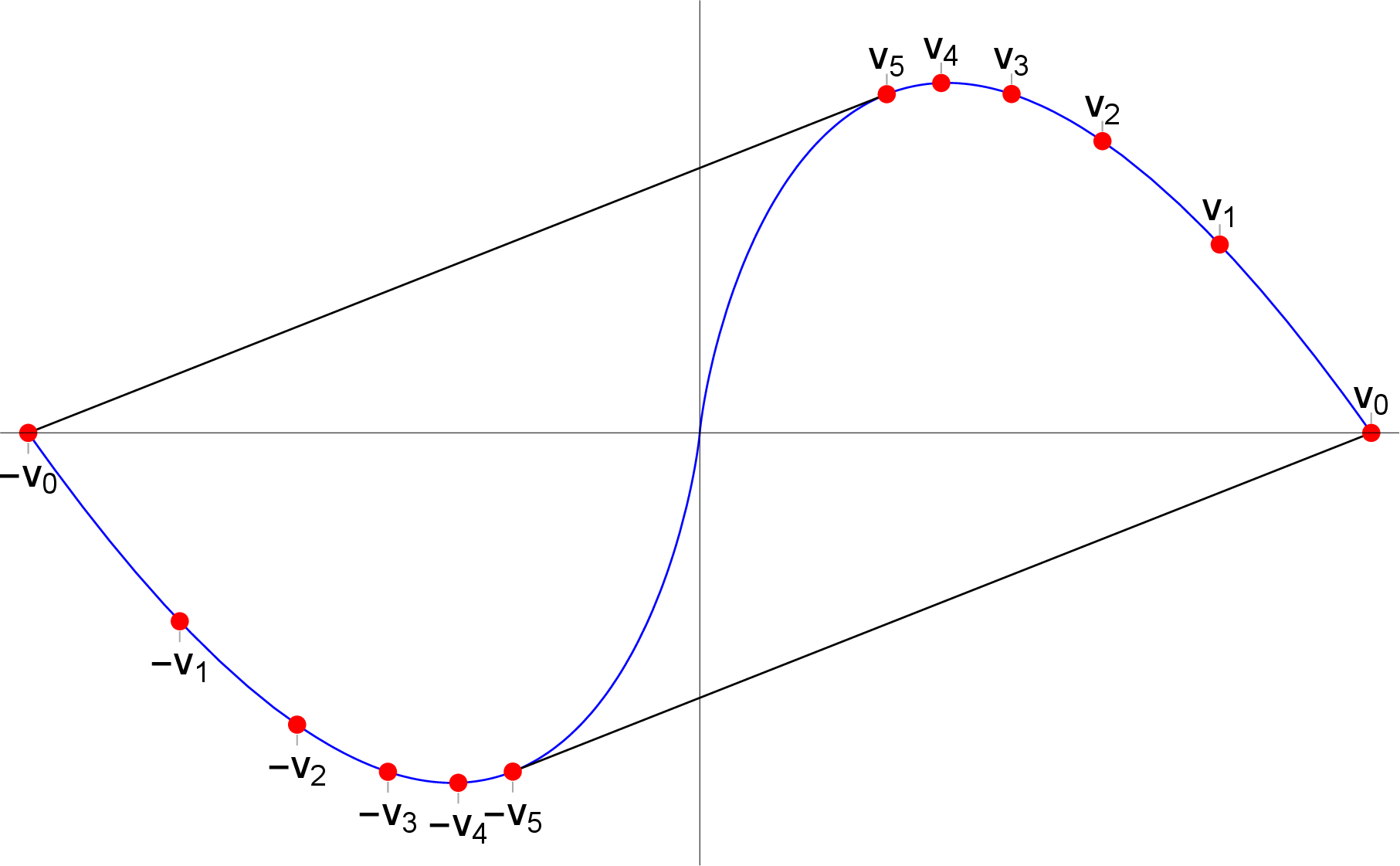}
\caption{For $n=5$, graph of $\frac{nx\log |x|}{\log c}$ (blue), the vectors (red) and tangent lines at $v_5,-v_5$ (black).}
\label{polygon_pic}
\end{figure}
By calculus, the tangent to the curve at the point $v_n$ passes through the point $\begin{bmatrix}
    -1 \\
    0
\end{bmatrix}=-v_0$. Since the function is strictly concave, the vectors $\pm v_0,\dots,\pm v_n$ are vertices of a centrally symmetric convex $(2n+2)$-gon $S$ (See \cref{polygon_pic}). Furthermore, $A_nv_n=v_{n+1}$ is in the interior of $S$ by strict concavity. This proves that $A_n(S)\subseteq S$ since all vertices map to another vertex except $v_n$, which maps strictly inside $S$.

Now consider a line $L\subseteq \mathbb{R}^2$ such that $L \cap S=\{v_n\}$. Let $B_n$ be the rank one linear map such that $B_nv_n=v_0$ and $\text{Ker}B_n$ is parallel to $L$. This choice of $B_n$ keeps $S$ invariant since we have the following inclusions: 
\[B_n(S)=B_n([-v_n,v_n])=[-v_0,v_0]\subseteq{S}\]

Given any vertex of $S$, its image under $A_n$ or $B_n$ either maps inside the polygon or is another vertex. The graph below shows vertices that map to other vertices and through what maps.
\begin{equation}\label{e.graph}
\begin{tikzpicture}[scale=1.15,baseline=(current  bounding  box.center)]
\tikzset{vertex/.style = {shape=circle,minimum size=0em}}
\tikzset{edge/.style = {->,> = latex'}}
\node[vertex] (za) at  (0,0) {$v_{0}$};
\node[vertex] (zb) at  (1,1) {$v_{1}$};
\node[vertex] (ze) at  (2.4,1) {};
\node[vertex] (zi) at  (2.4,-1) {};
\node[vertex] (zh) at  (1,-1) {$v_{n}$};
\draw[edge] (za) to["$A_n$"] (zb);
\draw[edge] (zb) to["$A_n$"]  (ze);
\draw[edge] (zi) to["$A_n$"]  (zh);
\draw[edge] (zh) to["$B_n$"] (za);
\draw [<-,dashed] (2.4,-1) arc[x radius=1, y radius=1, start angle=-90, end angle=90];
\end{tikzpicture}
\end{equation}

Let $\|\cdot\|_S$ be the Minkowski norm of the polygon $S$. Since $\|A_n\|_S=1$ and $\|B_n\|_S=1$, $\JSR(A_n,B_n)\leq 1$ by \cref{three member}. However, $\rho(A_n^nB_n)=1$ by construction, thus $\rho(A_n^nB_n)=\JSR(A_n,B_n)$ and so $A_n^nB_n$ is an SMP. 

To check uniqueness of the SMP, let $\Pi$ be any primitive product that is not a cyclic permutation of $A_n^nB_n$. Then any vertex of $S$ is mapped to the interior of $S$ by some power of $\Pi$; this follows from the structure of the graph \eqref{e.graph}. In particular, there exists $l$ such that $\Pi^l(S)\subseteq \mathrm{int}(S)$ and therefore $\|\Pi^l\|_S< 1$. This shows that $\rho(\Pi)<1$ and so $\Pi$ cannot be an SMP.

If $(A,B)$ is close enough to $(A_n,B_n)$ and re-scaled so that $\rho(A^nB)=1$, then $BA^n$ has an eigenvector $\tilde{v}_0=BA^n\tilde{v}_0$ close to $v_0$. For $i \geq 0$, let $\tilde{v}_i\coloneq A^i\tilde{v}_0$. For each $i$, $\tilde{v}_i$ must be close to $v_i$. So, for sufficiently small perturbation, the vectors $\pm \tilde{v}_0,\dots,\pm \tilde{v}_n$ are the vertices of a convex polygon $\tilde{S}$. Note that the combinatorics \eqref{e.graph} are unchanged, so $\tilde{S}$ is an invariant polygon for $(A,B)$ and the previous arguments persist. Therefore $A^nB$ is still the unique SMP.
\end{proof}

\section{ What are the SMPs if $(A,B) \in \mathcal{R}_{neg}$?}
In this section we will assume $A,B$ both have negative determinant and thus both $A,B$ are real diagonalizable.
\begin{lemma}\label{elliptic lemma}
If $A,B$ are matrices with negative determinants and $AB$ is not real diagonalizable, then $(A,B)$ is crossing or reducible.
\end{lemma}
\begin{proof}
To simplify notation, we will be using the associated $5$-tuple as defined in \eqref{e.5_var}. We may assume $x\geq 0,y\geq 0$ without loss of generality by independently rescaling $A,B$ if necessary.
If $AB$ is not real diagonalizable then $z^2-4uv\leq 0$ and thus $-2\sqrt{uv}\leq z\leq 2\sqrt{uv}$. However, this implies
\[4 u v-u y^2-v x^2+x y z-z^2\geq -u y^2-v x^2-2x y\sqrt{uv}=(\sqrt{-u}y-\sqrt{-v}x)^2\geq 0 \]
Hence
\[4 u v-u y^2-v x^2+x y z-z^2\geq 0\]
Which means $(A,B)$ is crossing or reducible by combining \cref{friedland,char. of red.,crossing}.
\end{proof}
We are now ready to prove the 3\textsuperscript{rd} part of \cref{intro thm}.
\begin{theorem}\label{both negative determinants}
If both $A,B$ have negative determinants, the pair is irreducible and $A^2\not=\JSR(A,B)^2I$, $B^2\not=\JSR(A,B)^2I$, then all the SMPs are among $A,B,AB$.
\end{theorem}
\begin{proof}
Without loss of generality we will normalize both $A,B$ so that $\JSR(A,B)=1$. If $(A,B)$ is crossing we are done by \cref{crossing smps}, thus we may assume $(A,B)$ is not crossing. By \cref{elliptic lemma}, $AB$ must be real diagonalizable. Now by applying \cref{lemma cross} twice (Interchanging the role of $A$ and $B$), we see that both $(A,AB)$ and $(B,AB)$ are crossing. However, by applying \cref{opposite determinants} to the pair $(A,AB)$ and $(B,AB)$ we see that all SMPs must be of the form $AB^n$ or $B$ as well as of the form $A^nB$ or $A$, unless $A$ or $B$ have finite order (Since $\JSR(A,B)=1$). However, the only matrices that have finite order and negative determinant must be conjugate to reflection and so have order $2$, which by assumption is not the case. Thus all the SMPs must be either $A,B,AB$ as those are the only products the two lists above have in common.
\end{proof}
\section{Ergodic Optimization and Joint Spectral Radius}\label{section 2}
To tackle the co-parallel case we need to view the joint spectral radius as a quantity related to ergodic optimization. Given a finite set of matrices $\mathcal{A}=\{A_0,\dots,A_{k-1}\}$, let $\sigma$ be the shift on $\Sigma_k=\{0,\dots,k-1\}^\mathbb{N}$. We define the following map for $\omega=(\omega_i)_{i \in \mathbb{N}}$
\begin{equation}
\mathcal{L}(\omega,n)=A_{\omega_n}\cdots A_{\omega_1}
\end{equation}
for which we may define the Lyapunov exponent.
\begin{definition}\label{e.lyap}
Let $\mathcal{M}_\sigma$ be the set of shift invariant probability measures. For each $\mu \in \mathcal{M}_\sigma$, the \emph{Lyapunov exponent} of $\mathcal{L}$ with respect to $\mu$ is defined as
\begin{equation}
    \lambda(\mu)\coloneq \lim_{n \to \infty}\frac{1}{n}\int_{\Sigma_k}\log \|\mathcal{L}(\omega,n)\|d\mu(\omega).
\end{equation}
\end{definition}
The joint spectral radius and the Lyapunov exponent of $\mathcal{L}$ are connected by the following abbreviated version of a theorem by Morris:
\begin{theorem}\cite{Morris_Mather}
\begin{equation}
\log\JSR(\mathcal{A})=\sup_{\mu \in \mathcal{M_\sigma}}\lambda(\mu)
\end{equation}
The supremum is always attained for some ergodic measure.
\end{theorem}
Any measure $\mu$ attaining the supremum is called a \emph{maximizing measure}.
\begin{note}
Each SMP gives rise to a maximizing measure supported on a periodic orbit. Two SMPs give rise to the same maximizing measure if and only if one is a cyclic permutation of the other.
\end{note}
To define Sturmian maximizing measures, we first define Sturmian sequences:
\begin{definition}
Given $\gamma,\rho \in [0,1]$ we define two sequences in $\Sigma_2$: \begin{align*} 
(\omega_{\gamma,\rho}^-)_n&=\lfloor \gamma (n+1) + \rho \rfloor-\lfloor \gamma n + \rho \rfloor\ \\
(\omega_{\gamma,\rho}^+)_n&=\lceil \gamma (n+1) + \rho \rceil-\lceil \gamma n + \rho \rceil
\end{align*}
The sequence $\omega_{\gamma,\rho}^-$ is called \emph{lower Sturmian} and $\omega_{\gamma,\rho}^+$ is called \emph{upper Sturmian} 
  of slope $\gamma$ and intercept $\rho$.
\end{definition}
\begin{note}
These are also called \emph{mechanical words}. See \cite[p.53]{Lothaire} for a more details.
\end{note}
The theorem below is, to our knowledge, folklore and describes all the properties of Sturmian sequences and measures we will need. Most of this theorem is proved in detail in \cite{finiteness}.\begin{theorem}[Folklore]\label{Sturmian}
Denote $X_\gamma$ as the set of all Sturmian sequences of slope $\gamma$. Then we have the following properties:
\begin{enumerate}
    \item $X_\gamma$ is a compact and invariant under the shift.
    \item The shift is uniquely ergodic on $X_\gamma$. This measure $\mu_\gamma$ is called the Sturmian measure of parameter $\gamma$.
    \item The map $\gamma \mapsto \mu_\gamma$ where $\gamma \in [0,1]$ is continuous. Hence the set of all Sturmian measures is compact.
\end{enumerate}
\end{theorem}
\section{What are the possible SMPs if $(A,B) \in \mathcal{R}_{copar}$ ?}\label{section 9}
In this section, we describe the possible SMPs given that $(A,B)$ is co-parallel. We remind the reader, that given a word $w \in \{0,1\}^*$ and matrices $A,B$, we denote $w(A,B)$ as the matrix product given by replacing each $0$ with $A$ and $1$ with $B$. We begin this section with introducing special words that come from Sturmian sequences.
\begin{definition}\label{e. def of christoffel}
Let $\frac{p}{q} \in\mathbb{Q}\cap[0,1]$ with $\gcd(p,q)=1$ and $\rho\in[0,1]$.
\begin{itemize}
    \item $q$-prefix of $\omega_{\frac{p}{q},\rho}^-$ is called a \emph{Sturmian} word
    \item $q$-prefix of $\omega_{\frac{p}{q},0}^-$ is called a \emph{Christoffel} word
\end{itemize}
\end{definition}
\begin{note}
All Sturmian words are primitive. Christoffel words are exactly the Sturmian words that are Lyndon. When we refer to a Sturmian/Christoffel product we mean a product $w(A,B)$ where $w$ is a Sturmian/Christoffel word. We will similarily refer to SMPs being Sturmian/Christoffel.
\end{note}
All Christoffel words can be iteratively obtained using the \emph{Christoffel tree} which is constructed as follows. Each node will be a pair of words. The root is $(0,1)$. Each node $(u,v)$ has two children, the left child is $(u,uv)$ and the right child is $(uv,v)$( See \cref{Chris tree} for the first three layers of the tree.).
\begin{theorem}\cite{BerstelLuca}\label{chris cop}
The Christoffel tree contains exactly once the decomposition $(u,v)$ of each Christoffel word $w=uv$ of length greater than $1$. All words in each node are Christoffel.
\end{theorem}

\Tree [.$(0,1)$ [.$(0,01)$ [.$(0,001)$  {\ldots} {\ldots} ] [.$(001,01)$  {\ldots} {\ldots} ] ]  [.$(01,1)$ [.$(01,011)$ {\ldots} {\ldots} ] [.$(011,1)$  {\ldots} {\ldots} ] ] ]
\begin{figure}[h]
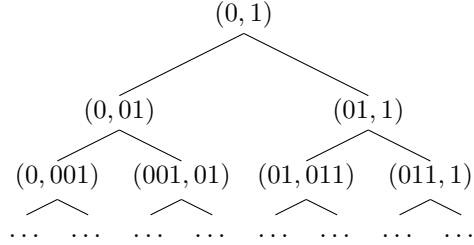

    \caption{The Christoffel tree}
    \label{Chris tree}
\end{figure}

We now prove a few lemmas that will be needed in the main theorem.
\begin{lemma}\label{product of coparallel}
If $(A,B)$ is co-parallel, then $AB$ is co-parallel to both $A$ and $B$, with its axis lying between the axis of $A$ and of $B$. Furthermore $(A^n,B^k)$ is also co-parallel for any $n,k\geq 1$ with the same axis as $A$ and $B$.
\end{lemma}
\begin{proof}
The first part is proved in \cite[Lemma 4.2]{JorgensenSmith}. The second part follows from the fact that $A^n$ and $B^k$ have exactly the same eigenvectors as $A$ and $B$. The only way that would not be the case is if $A^n$ or $B^k$ are multiples of identity. However, that is impossible since that would mean $A$ or $B$ are conjugate to a multiple of a rational rotation and therefore not real diagonalizable.
\end{proof}
 To make the notation more manageable in the following corollary $u$ should be understood as the product $u(A,B)$.
\begin{corollary}\label{all coparallel}
Let $(A,B)$ be co-parallel. Then a Christoffel product $w(A,B)$ is co-parallel to any other 
Christoffel product $w'(A,B)$.
\end{corollary}
\begin{proof}
We show any product the Christoffel tree is co-parallel to any other product in the tree. 
Using the \cref{product of coparallel}, it follows by induction that given any node $(u,v)$, for all of its descendants $(u',v')$, $u'$ (resp. $v'$)  is co-parallel with both $u$ and $v$ with its axis laying between the axis of $u$ and $v$.

Now take any two distinct nodes $(u_1,v_1)$ and $(u_2,v_2)$ (We assume one is not a descendant of another, similar argument works here as well) and trace them back to the last common ancestor say $(u',v')$. Without loss of generality, $(u_1,v_1)$ is a descendant of $(u',u'v')$ and $(u_2,v_2)$ is a descendant $(u'v',v')$, otherwise there would be a deeper node in common. However, the axis of $u_1v_1$ must be between both axis of $u',u'v'$ and axis of $u_2v_2$ must be between both axis of $u'v',v'$. Therefore, they must be disjoint. Hence $u_1v_1$ and $u_2v_2$ must be co-parallel.
\end{proof}
\begin{lemma}\label{e.lyap cont}
Let $(A_0,A_1)$ be co-parallel. Then the associated Lyapunov exponent $\lambda: \mathcal{M}_\sigma \to \mathbb{R}$, as defined in \cref{e.lyap}, is continuous in the weak* topology.
\end{lemma}
\begin{proof}
We may assume $A_0,A_1$ have non-negative trace by possibly negating the matrices. Since they have non-negative trace and positive determinants, they have positive eigenvalues. Hence, there exists a strictly forward invariant cone $C$ such that $A_0(C)\subset C$ and $A_1(C)\subset C$; any cone that contains the attracting eigenvectors in its interior and does not contain the repelling eigenvectors will work. Existence of such a cone is equivalent to the associated cocycle $\mathcal{L}(\omega,n)$ being dominated (see \cite{JG} for proof and definitions). Dominance guarantees existence of a continuous map $s: \Sigma_2 \to \mathbb{R}^2$ such that for any $\omega \in \Sigma_2$ $\|s(\omega)\|=1$ and:
\[\frac{\mathcal{L}(\omega,n)s(\omega)}{\|\mathcal{L}(\omega,n)s(\omega)\|}=s(\sigma^n(\omega)).\]
Furthermore, using \cite[Corollary 2.4]{JairoRams} we know that given $\mu \in \mathcal{M}_\sigma$, for $\mu$ almost every $\omega$:
\[ \lim_{n \to \infty}\frac{1}{n}\log \|\mathcal{L}(\omega,n)\|=\lim_{n \to \infty}\frac{1}{n}\log \|\mathcal{L}(\omega,n)s(\omega)\|\]
Therefore, integrating and applying the dominated convergence theorem:
\[\lambda(\mu)=\int_{\Sigma_2}\lim_{n \to \infty}\frac{1}{n}\log \|\mathcal{L}(\omega,n)s(\omega)\|d\mu(\omega).\]
Using the properties of $s(\omega)$ we may write:
\[\lim_{n \to \infty}\frac{1}{n}\log \|\mathcal{L}(\omega,n)s(\omega)\|=\lim_{n \to \infty}\frac{1}{n} \sum_{i=0}^{n-1} \log \|\mathcal{L}(\sigma^j(\omega),1)s(\sigma^j(\omega))\|\]
Now, by Birkhoff ergodic theorem:
\[\lambda(\mu)=\int_{\Sigma_2}\log{\|\mathcal{L}(\omega,1)s(\omega)\|}d\mu(\omega)\]
Since $\log{\|\mathcal{L}(\omega,1)s(\omega)\|}$ is continuous, $\lambda$ is continuous with respect to the weak* topology.
\end{proof}
We now state another theorem of Jorgensen-Smith, we remind the reader that $W(a,b)$ denotes the set of products that have $a$ $A$'s and $b$ $B$'s.
\begin{theorem}\cite{JorgensenSmith}\label{coparallel}
If $(A,B)$ is co-parallel, then for any pair of positive integers $a$ and $b$, the products that maximize the spectral radius in $W(a,b)$ are cyclic permutations of the prefix of length $a+b$ of the product $\omega_{\frac{b}{a+b},0}^-(A,B)$.
\end{theorem}
We note that Jorgensen and Smith proved this result under the additional hypothesis $(A,B) \in SL(2,\mathbb{R})^2$ and $\tr{A},\tr{B}>2$. However, the more general statement follows by re-scaling the matrices and noting that all products in $W(a,b)$ have the same determinant. Also, Jorgensen and Smith maximize the trace, but under the hypothesis this is equivalent to maximizing the spectral radius.

We are finally ready to prove 4\textsuperscript{th} part of \cref{intro thm}.
\begin{theorem}
Let $(A,B)$ be co-parallel. There always exists exactly one Sturmian maximizing measure. If an SMP exists, it must be unique (up to cyclic permutations) and be Sturmian.
\end{theorem}

\begin{proof}

We first show that there is always exactly one Sturmian maximizing measure. By \cref{e.BW} we know that the maximizing measures can be approximated arbitrarily well by periodically supported measures. In this context, \cref{coparallel} says that the Lyapunov exponent over periodic measures of a given period is maximized for some Sturmian measure of a rational parameter. Denote $\mathcal{M}_\sigma$ as all shift invariant probability measures, $\mathcal{M}^p_\sigma$ as all shift invariant probability measures supported on a periodic orbit and $\mathcal{M}^{\mathbb{Q}}_\sigma$ as all Sturmian measures of rational parameter. Combining the two results we get the following chain of equalities. \[\sup_{\mu\in\mathcal{M}_\sigma}\lambda(\mu)\underset{\cref{e.BW}}{=}\sup_{\mu\in \mathcal{M}^p_\sigma}\lambda(\mu)\underset{\cref{coparallel}}{=}\sup_{\mu\in\mathcal{M}^{\mathbb{Q}}_\sigma}\lambda(\mu)\]
However, by \cref{Sturmian} the set of all Sturmian measures is compact. Therefore we know there is a convergent sequence of Sturmian measures of rational parameter $\mu_n\to\mu$ that converges to a maximizing Sturmian measure. 

To show uniqueness, we prove that the map $f(t)=\lambda(\mu_t)$ is strictly concave on $[0,1]$, where $\mu_t$ denote the Sturmian measure of parameter $t$. Given any two rational numbers $t_1,t_2\in \mathbb{Q}$, they can be written with a common denominator, $t_1=\frac{p_1}{q}, t_2=\frac{p_2}{q}$. Let $w_1$ be the prefix of length $q$ of the Sturmian sequence $\omega_{t_1,0}^-$ and $w_2$ be the prefix of length $q$ of the Sturmian sequence $\omega_{t_2,0}^-$. Lastly, we let $w_3$ be the prefix of length $2q$ of the Sturmian sequence $\omega_{\frac{t_1+t_2}{2},0}^-$. Notice that $w_1,w_2,w_3$ are all powers of Christoffel words. Denote $\Pi_1=w_1(A,B), \Pi_2=w_2(A,B) , \Pi_3=w_3(A,B)$. By \cref{all coparallel,product of coparallel} $(\Pi_1,\Pi_2)$ is co-parallel. Applying \cref{alt}:
\[\rho(\Pi_1\Pi_2)>\rho(\Pi_1)\rho(\Pi_2)\]
Furthermore, by \cref{coparallel}:
\[\rho(\Pi_3)\geq \rho(\Pi_1\Pi_2) \]
Combining the two inequalities we get
\[\rho(\Pi_3)>\rho(\Pi_1)\rho(\Pi_2)\]
Taking both sides to the power $\frac{1}{2q}$ and taking the logarithm we get:
\[\log(\rho(\Pi_3)^\frac{1}{2q})>\frac{\log(\rho(\Pi_1)^\frac{1}{q})}{2}+\frac{\log(\rho(\Pi_2)^\frac{1}{q})}{2}\]
The above inequality can be re-written in terms of $f$: 
\[f\left(\frac{p_1+p_2}{2q}\right)>\frac{f\left(\frac{p_1}{q}\right)+f\left(\frac{p_2}{q}\right)}{2}.\]
Which shows that $f$ is strictly mid-point concave on the rationals. Since $f$ is continuous by \cref{e.lyap cont}, it must be strictly mid-point concave everywhere. Hence, $f$ is strictly concave. Any continuous strictly concave function on a compact set has exactly one maximum. This proves that there is exactly one Sturmian maximizing measure.

We are now ready to show that if an SMP exists, then it is Sturmian and is unique up to cyclic permutations. We know by \cref{coparallel} that all the SMPs must be Sturmian since they have to maximize spectral radius in their respective set $W(a,b)$. Therefore, they all must come from some Sturmian maximizing measure of rational parameter. However, these are unique, so there can only be one such SMP up to cyclic permutations.
\end{proof}

\begin{note}
We suspect that the maximizing measure is always unique (and therefore Sturmian) in the co-parallel region. We also suspect that the counterexamples to the finiteness conjecture form a subset of zero Hausdorff dimension. Some results in this direction were given by \cite{finiteness,Hare,JenkinsonPollicott}.
\end{note}

\section{Generic Uniqueness}\label{generic uniqueness}
In this section, we prove that in the union of the three regions $\mathcal{R}_{cross}\cup\mathcal{R}_{mix}\cup\mathcal{R}_{neg}$, the SMPs are generically unique and of the form given in \cref{intro thm}. We first prove a few lemmas. We denote $\lambda_n$ as the $n$-dimensional Lebesgue measure.
\begin{lemma}
Consider a polynomial map $\pi:\mathbb{R}^n\to \mathbb{R}^m$ such that $\lambda_m(\pi(\mathbb{R}^n))>0$. Then for any Lebesgue measurable set $E\subseteq \mathbb{R}^m$,
\[\lambda_m(E)=0 \implies \lambda_n(\pi^{-1
}(E))=0  \]
\end{lemma}
\begin{proof}
Denote $C_\pi \subseteq \mathbb{R}^n$ as the critical points of $\pi$. Since the map $\pi$ is a polynomial, $C_\pi$ is algebraic and so either $\lambda_n(C_\pi)=0$ or $C_\pi=\mathbb{R}^n$. By Sard's theorem $\lambda_m(\pi(C_\pi))=0$. Since $\lambda_m(\pi(\mathbb{R}^n))>0$, $C_\pi$ must have zero Lebesgue measure. This means $\pi$ is almost everywhere a local submersion and now the lemma follows by the rank theorem and Fubini, for details see \cite{Ponomarev}.
\end{proof}
We apply the above lemma to the mapping: \[(A,B)\mapsto (\tr{A},\tr{B},\tr{AB},\det{A},\det{B}).\]
By \cref{real-realizable} its image has positive Lebesgue measure. Thus, if we want to prove a generic statement about pairs of real matrices $(A,B)$, we may equivalently prove it for generic $5$-tuples. We will tacitly do so from now on, without repeatedly invoking the above lemma.
\begin{lemma}\label{generic smp form}
For Lebesgue almost every pair $(A,B) \in \mathcal{R}_{cross}\cup\mathcal{R}_{mix}\cup\mathcal{R}_{neg}$, all the SMPs (up to cyclic permutation) are of the form given by \cref{intro thm}.
\end{lemma}
\begin{proof}
If $(A,B) \in \mathcal{R}_{cross}$, all the SMPs are either $A$ or $B$ by \cref{crossing smps} so there is nothing left to prove.

If $(A,B) \in \mathcal{R}_{mix}$, we may assume that $A,B$ are invertible as non-invertible matrices form a null set. Now by \cref{ mixed case} it is sufficient to show that pairs $(A,B)$ for which there exists $k\geq 1$ such that
\[A^k=\JSR(A,B)^kI,\]
form a measure zero set. Assume $(A,B)$ are such that the above holds, then $A$ must be conjugate to a multiple of a rotation by $2\pi\theta$ where $\theta\in \mathbb{Q}$. Such matrices satisfy the equation $\det(A)=\frac{(\tr{A})^2}{4\cos^2(2\pi\theta)}$. Thus, for fixed $\theta$, all associated $5$-tuples live on the on the parametric family \[M_\theta=(x,y,z,\frac{x^2}{4\cos^2(2\pi\theta)},v)\]
which has Lebesgue zero measure in $\mathbb{R}^5$. Since the set of rational numbers is countable the set $\cup_{\theta \in \mathbb{Q}}M_\theta$ also has measure zero. Thus, such $(A,B)$ must form a measure zero set.

If $(A,B) \in \mathcal{R}_{neg}$, then by \cref{both negative determinants} it is sufficient to show that pairs $(A,B)$ such that
\[A^2=\JSR(A,B)^2I\quad\text{or 
}\quad B^2=\JSR(A,B)^2I  \]
form a measure zero set. Without loss of generality assume $A$ satisfies the above, then $A$ must be conjugate to a multiple of a reflection. Such matrices satisfy the equation $\tr{A}=0$. Hence, all associated $5$-tuples live in 
\[(0,y,z,u,v)\]
which is measure zero. Similarly if $B$ is conjugate to reflection we get a zero measure set. Hence, such $(A,B)$ must form a measure zero set.

\end{proof}
\begin{lemma}\label{generic red}
The set of reducible pairs of matrices $(A,B)$ is a proper algebraic set. In particular, it is nowhere dense and has Lebesgue measure zero.
\end{lemma}
\begin{proof}
By \cref{friedland}, the $5$-tuples of reducible pairs of matrices are exactly the $5$-tuples that satisfy the equation
\[4 u v-u y^2-v x^2+x y z-z^2= 0.\]
Therefore  $5$-tuples of reducible pairs of matrices are the pre-image of $\{0\}$ of a non-zero polynomial.
\end{proof}
\begin{lemma}\label{l.ae}
For Lebesgue almost every pair $(A,B)$ of real $2 \times 2$ matrices, if $\Pi$ is an SMP that has non-real eigenvalues, then $\Pi=A$ or $\Pi=B$ and the SMP is unique.
\end{lemma}

\begin{proof}
By $\cref{generic red}$ we may assume $(A,B)$ is irreducible. 
For those pairs, $\varrho \coloneqq \mathrm{JSR}(A,B)$ is nonzero and there exists a centrally symmetric convex body $C \subseteq \mathbb{R}^2$ such that the convex hull of $A(C) \cup B(C)$ is $\varrho C$ (See \cite{Protasov96}).
Suppose $\Pi$ is an SMP with non-real eigenvalues.
Then, $\det \Pi = \varrho^{2k}$, where $k\coloneqq |\Pi|$, so the inclusion $\Pi(C) \subseteq \varrho^k C$ is actually an equality. 
It follows that either $A(C)$ or $B(C)$ equals $\varrho C$. 
Almost surely we have $\det A \neq \det B$, and in this case $A(C) \neq B(C)$.
Since $\Pi$ is primitive, it must be either $A$ or $B$, and $\Pi$ is the unique SMP.
\end{proof}

Given a word $w\in \{0,1\}^*$, define its signature $(m,k,l)$ as follows: we first cyclically permute $w$ so that it is a Lyndon word and the three numbers are: $m$ is the number of $0$'s, $k$ is the number of $1$'s, and $l$ is the number of $01$'s that can be found within that new word. We remind the reader that given a word $w$, $F_{w}(x,y,z,u,v)$ denotes the Fricke polynomial described in \cref{e. fricke}.

\begin{lemma}\label{independent fricke}
Let $w$ and $w'$ be words whose signatures $(a,b,c)$ and $(a',b',c')$ are linearly independent, then the polynomials $F_w(x,y,z,0,0), F_{w'}(x,y,z,0,0)$ are algebraically independent. That is, there is no non-zero polynomial $P$ such that:
\[P(F_w(x,y,z,0,0),F_{w'}(x,y,z,0,0))=0\]
\end{lemma}
\begin{proof}
Consider a pair of matrices $A,B$ associated to the $5$-tuple $(x,y,z,0,0)$ where $x,y,z$ are free. Now we use observations similar ones in \cref{non-inv}.  Cyclically permuting $w$ if necessary,
$w(A,B)$ can be written as
\[w(A,B)=A^{a_1}B^{b_1}\cdots A^{a_c}B^{b_c} \]
with $a_i,b_i\geq1$.
Since $A,B,AB$ are all non-invertible, by Cayley-Hamilton we have
\[A^n=(\tr{A})^{n-1}A,\quad B^n=(\tr{B})^{n-1}B,\quad (AB)^n=(\tr{AB})^{n-1}AB.\] Thus, we may reduce all powers of $A$ and $B$ by pulling out traces:
\[w(A,B)=(\tr{A})^{a-c}(\tr{B})^{b-c}(AB)^c\]
We now take the trace and get the following expression:
\[\tr{w(A,B)}=(\tr{A})^{a-c}(\tr{B})^{b-c}(\tr{AB})^c=x^{a-c}y^{b-c}z^c\]
By exactly the same reasoning:
\[\tr{w'(A,B)}=(\tr{A})^{a'-c'}(\tr{B})^{b'-c'}(\tr{AB})^{c'}=x^{a'-c'}y^{b'-c'}z^{c'}\]
Now if the polynomials $F_{w}(x,y,z,0,0),F_{w'}(x,y,z,0,0)$ were algebraically dependent then:
\[F_{w}(x,y,z,0,0)=x^{a-c}y^{b-c}z^c
\text{ and  }
F_{w'}(x,y,z,0,0)=x^{a'-c'}y^{b'-c'}z^{c'}\]
would also have to be dependent. However, since $(a,b,c)$ and $(a',b',c')$ are linearly independent, we arrive at a contradiction.
\end{proof}

We are now ready to prove generic uniqueness.
\begin{theorem}
Almost every pair $(A,B) \in \mathcal{R}_{cross}\cup\mathcal{R}_{mix}\cup\mathcal{R}_{neg}$ has a unique SMP that is of the form described in \cref{intro thm}.
\end{theorem}
\begin{proof}
Assume for a contradiction that there exists a subset $\mathcal{P}$ of $\mathcal{R}_{cross}\cup\mathcal{R}_{mix}\cup\mathcal{R}_{neg}$ with positive Lebesgue measure such that every pair in $\mathcal{P}$ admits two SMPs. We know for almost every pair $(A,B)$ all SMPs are of the form $A^nB$ or $AB^n$ by \cref{generic smp form}. Since the set of such products is countable, we can find a pair of distinct primitive words $w$, $w'$ of the form $0^n1$ or $01^n$ and a subset $\mathcal{Q} \subseteq \mathcal{P}$ of positive measure such that for every $(A,B) \in \mathcal{Q}$, $w(A,B)$ and $w'(A,B)$ are both SMPs.
By \cref{l.ae} we may reduce the set $\mathcal{Q}$ if necessary, and assume $w(A,B),w'(A,B)$ have real eigenvalues for all $(A,B) \in \mathcal{Q}$.

Fix $(A,B) \in \mathcal{Q}$ and consider the two distinct SMPs $\Pi=w(A,B)$ $\Pi'=w'(A,B)$ with lengths $k \coloneqq |\Pi|$, $k' \coloneqq |\Pi'|$.
Let $\lambda_1$, $\lambda_2$ (resp.\ $\lambda_1'$, $\lambda_2'$) be the eigenvalues of $w$ (resp.\ $w'$). By assumption, all eigenvalues $\lambda_1,  \lambda_2, \lambda_1', \lambda_2'$ are real. 
Furthermore,
\begin{equation}
\mathrm{JSR}(A,B) 
= \max\left\{|\lambda_1|^{\frac{1}{k}},|\lambda_2|^{\frac{1}{k}}\right\}
= \max\left\{|\lambda_1'|^{\frac{1}{k'}},|\lambda_2'|^{\frac{1}{k'}}\right\} \, .
\end{equation}
Since eigenvalues are real,
\begin{equation}
\big( \lambda_1^{2k'} - (\lambda'_1)^{2k} \big) 
\big( \lambda_1^{2k'} - (\lambda'_2)^{2k} \big) 
\big( \lambda_2^{2k'} - (\lambda'_1)^{2k} \big) 
\big( \lambda_2^{2k'} - (\lambda'_2)^{2k} \big) = 0 \, .
\end{equation}
Expanding and using Newton--Girard identities, this equation becomes:
\begin{equation}
P(\tr \Pi, \tr \Pi' , \det \Pi, \det \Pi') = 0 \, ,
\end{equation}
where $P$ is some polynomial in four variables, with real coefficients.
Letting as usual $x= \tr A$, $y = \tr B$, $z = \tr AB$, $u =\det A$, $v=\det B$, the relation above can be rewritten
\begin{equation}\label{e.relation}
P\big( F_w(x,y,z,u,v), F_{w'}(x,y,z,u,v), u^av^b,u^{a'}v^{b'}) = 0 
\end{equation}
where $F_w,F_{w'}$ are the Fricke polynomials and $a,b,a',b'$ are appropriate non-negative integers.

Equation \eqref{e.relation} holds for all $(A,B)$ in a positive measure subset of pairs $(A,B)$.
It follows that the same equation, seen in terms of variables $x,y,z,u,v$, holds on a positive measure subset of $\mathbb{R}^5$. Since proper algebraic sets have zero measure, we conclude that the equation holds on all of $\mathbb{R}^5$. Setting $u=v=0$ we obtain:
\[P\big( F_w(x,y,z,0,0), F_{w'}(x,y,z,0,0),0,0) = 0. \] However, since $w,w'$ are two distinct words of forms $0^n1$ or $10^n$, they necessarily have linearly independent signatures. By \cref{independent fricke} such $P$ cannot exits and we arrive at a contradiction.
\end{proof}
With this theorem, we have completed the proof of \cref{intro thm}.
\section{The unexplored regions}
There are essentially two regions that remain unexplored $\mathcal{S}_{anti},\mathcal{S}_{complex}$:
\begin{itemize}
\item $\mathcal{S}_{anti}$ is the region where $(A,B)$ is anti-parallel.
\item $\mathcal{S}_{complex}$ is the region where $A$ or $B$ have complex eigenvalues.
\end{itemize}

Some notes about these regions:
\begin{enumerate}
    \item Generic uniqueness of SMPs (modulo only cyclic permutations) does not hold in $\mathcal{S}_{complex}$, see \cite{BL}. There are two cases, either both or only one matrix has complex eigenvalues. The examples in \cite{BL} are all in the latter category; however there exist examples where both matrices have complex eigenvalues. One such example is given by the $5$-tuple:
    \[(-3.76601, -0.49459, -8.13153, 3.52510, 8.71249).\]
    \item By our classification, generically, non-Sturmian SMPs can only occur in these unexplored regions. It is not too hard to find explicit examples of non-Sturmian SMPs in both of them.
    \item Numerically, non-Sturmians are quite rare in $\mathcal{S}_{anti}\cap SL(2,\mathbb{R})^2$. It seems that they are concentrated near the curve $(x,x,2-x^2,1,1)$. Note that any pair $(A,B$) near this curve has the following properties $\rho(A)\approx\rho(B)$ and $\rho(AB)\approx \rho(A)\rho(B)$
\end{enumerate}
\section{Non-negative matrices}
In the introduction, we claimed that our classification encompasses all non-negative matrices. We are now ready to prove this result. We begin with a lemma characterizing non-negative matrices with positive determinants.
\begin{lemma}\label{non-negative lemma}
Let $(A,B) \in GL^+(2,\mathbb{R})^2$ be an irreducible pair of real diagonalizable matrices, then $(A,B)$ is simultaneously conjugate to non-negative matrices if and only if $(A,B)$ is either crossing or co-parallel and $\tr(A),\tr(B)\geq 0$
\end{lemma}
\begin{proof}
We first prove the forward direction. We may assume our matrices $A,B$ are non-negative. By Perron-Frobenius theorem, both M\"obius transformations  $f_A$ and $f_B$ have an attracting non-negative fixed point and a repelling non-positive fixed point. Hence $(A,B)$ cannot be anti-parallel. Since $A,B$ are non-negative, it is clear that $\tr{A},\tr{B}\geq 0$ and so this direction follows.

 Now assume that $(A,B)$ is co-parallel or crossing with non-negative traces. Since the traces are non-negative and determinants are positive, both $A$ and $B$ have only positive eigenvalues. Therefore, in either case (co-parallel or crossing) there exists a cone $C$ such that $A(C)\subseteq C$ and $B(C)\subseteq C$ (Any cone that contains the dominant eigenvectors of $A$ and $B$ and does not contain the non-dominant eigendirections of $A$ and $B$ will work). Now take a matrix $P$ that maps the cone $C$ to the first quadrant and consider $(PAP^{-1},PBP^{-1})$. These new matrices keep the first quadrant invariant by construction. That means both of them must be non-negative as desired.
\end{proof}
We are now ready to prove the following corollary:
\begin{corollary}
Let $(A,B)$ be matrices that are simultaneously conjugate to non-negative matrices, then their SMP, if exists, must be Sturmian. 
\end{corollary}
\begin{proof}
By \cref{intro thm} if either $A$ or $B$ have a negative determinant, then an SMP exists and must be Sturmian since products of the form $A^nB$ and $AB^n$ are Sturmian. If both $A,B$ have positive determinants then by \cref{non-negative lemma} $(A,B)$ must be crossing or co-parallel. However, again invoking \cref{intro thm} we see that if an SMP exists, it must be Sturmian.
\end{proof}
\section*{Acknowledgements}
I would like to thank my advisor Jairo Bochi for invaluable suggestions and comments. This paper would not be possible without his insight and encouragement.I would also like to thank the anonymous referee. Their comments greatly improved readability and allowed me to fix a flawed proof.

\end{document}